\documentclass[a4paper]{article}
\usepackage[utf8]{inputenc}
\usepackage{authblk}
\usepackage[all]{xy}
\usepackage{amsmath,amssymb,amsthm,enumerate,graphicx}

\theoremstyle{plain}
 \newtheorem{theorem}{Theorem}[section]
 \newtheorem{corollary}[theorem]{Corollary}
 \newtheorem{lemma}[theorem]{Lemma}
 
 \theoremstyle{definition}
 \newtheorem{definition}[theorem]{Definition}
 \theoremstyle{remark}

 \numberwithin{equation}{section}

\newcommand{\Z}{\mathbb Z}

\newcommand{\R}{\mathbb R}
\newcommand{\C}{\mathbb{C}}
\newcommand{\F}{\mathbb{F}}

\newcommand{\gen}[1]{\langle #1 \rangle}
\newcommand{\inv}[1]{{#1^{-1}}}
\newcommand{\lcm}{\mathrm{lcm}}

\newcommand{\checkk}[1]{#1{\check\ \check\ }}

\newcommand{\ot}{\leftarrow}

\newcommand{\mapsfrom}{\mathrel{\reflectbox{\ensuremath{\mapsto}}}}

\newcommand{\Ast}{\mathop{\scalebox{1.2}{\raisebox{-0.1ex}{$\ast$}}}}

\newcommand{\GL}{\mathrm{GL}}
\newcommand{\SL}{\mathrm{SL}}
\newcommand{\PGL}{\mathrm{PGL}}
\newcommand{\PSL}{\mathrm{PSL}}

\newcommand{\ZG}{\mathbb{Z}G}
\newcommand{\ZY}{\mathbb{Z}Y}
\newcommand{\CG}{\mathbb{C}G}
\newcommand{\CY}{\mathbb{C}Y}

\newcommand{\Hom}{\mathrm{Hom}}

\newcommand{\Aut}{\mathrm{Aut}}

\newcommand{\Fix}{\mathrm{Fix}}

\newcommand{\im}{\mathrm{im}~}

\newcommand{\ZU}{Z_{\mathrm{u}}}

\newcommand{\GU}{\Gamma_{\mathrm{u}}}

\renewcommand{\c}{\mathbf{c}}
\renewcommand{\v}{\mathbf{v}}
\renewcommand{\o}{\mathbf{o}}
\newcommand{\e}{\mathbf{e}}
\renewcommand{\r}{\mathbf{r}}
\newcommand{\s}{\mathbf{s}}
\newcommand{\f}{\mathbf{f}}
\newcommand{\p}{\mathbf{p}}
\newcommand{\q}{\mathbf{q}}
\renewcommand{\b}{\mathbf{b}}
\renewcommand{\t}{\mathbf{t}}
\newcommand{\genus}{\mathfrak{g}}
\newcommand{\Sph}{\mathbb{S}}
\newcommand{\Disk}{\mathbb{D}}

\title{Covering groups of minimal exponent}
\author{Nicola Sambonet}

\affil{Instituto de Matemática e Estatística --
Universidade Federal da Bahia\\
Avenida Adhemar de Barros s/n, 40170-110, 
Salvador, Brazil\\
e--mail: nsambonet@gmail.com}

\begin{document}

\maketitle
\begin{abstract}
\noindent
  Presenting a finite group by a free product of finite cyclic groups the Hopf formula for the Schur multiplier affords also a covering group, and this has minimal exponent provided that the order of the generators is preserved.
  This condition corresponds to a covering projection between simplicial complexes, and so a presentation by a Fuchsian group corresponds to a covering projection between compact surfaces.
\end{abstract}

\section{Introduction}
From the point of view of representation theory, a \emph{cover} of a finite group is a finite central extension $1\to A\to E\to G\to 1$ with the property that all the projective representations of $G$ can be lifted to ordinary representations of $E$  \cite{Isaacs}.
In general the projective lifting problem is controlled by a homomorphism $\eta$ from the dual of $A$ into the Schur multiplier, that is the second cohomology group with complex coefficients $H^2G=H^2(G,\C^\times)$.
Thus the covers are the extensions such that $\eta$ is surjective.
If $\eta$ is an isomorphism then $E$ is a Schur cover, this is an extension of $G$ by a copy of $H^2G$ and so it has minimal order among the covers. 
A fundamental theorem of Schur asserts that a Schur cover always exists, and illustrates how to construct one provided that $H^2G$ is known \cite{Schur1904}.
Thus, in a second paper on the subject \cite{Schur1907}, Schur describes a formula that is by far the most efficient way to compute the multiplier, namely for any free presentation $1\to R\to F\to G\to 1$ we have
\[H^2G\simeq R\cap[F,F]/[R,F]\ .\]
However, the computational complexity of this formula is hard.
For instance it is only recently that, combining this formula with powers--commutators calculus and computer assistance, M.~Vaughan--Lee discovered the first examples of odd order such that the exponent of $H^2G$ does not divide the exponent of $G$. They come more than a century after Schur introduced the multiplier and half a century after A.J.~Bayes, J.~Kautsky, and J.W.~Wamsley found an analogue example for the case of 2-groups \cite{BaKaWa,VaughanLee2021}.
The formula can be proved by observing that $E=F/[R,F]$ maps onto any Schur cover, vice versa, there is a way to produce a Schur cover once that $E$ is known.
However, being an infinite group $E$ itself is not suitable for representation theory, and both this way and the original construction to obtain a Schur cover depend on a choice.
In turn there can be many non-isomorphic Schur covers, and they are far from being universal objects, with the exception of the perfect groups \cite{Aschbacher}.
This facts suggest to drift from the Schur covers, and the aim of this paper is to describe those covers which can be obtained by a more general version of the formula.
Indeed the homology theory offers a generalization, and in fact the rediscovery of the formula by H.~Hopf is a milestone in the birth of this subject, being the first evidence relating algebraic topology with representation theory \cite{Brown,Dieudonne,Hopf}.
Thus, the same formula describes the second homology group $H_2G=H_2(G,\Z)$ for any group $G$, and that for finite groups the two formulas coincide is nowadays clarified since the universal coefficient theorem gives an isomorphism $H^2G\simeq\Hom(H_2G,\C)$.
Moreover, for any group extension $1\to R\to F\to G\to 1$ there is an exact sequence
\begin{equation}\label{equation:LHS short exact sequence}
 H_2F\to H_2G\to(H_1R)_G\to H_1F\to H_1G\to 1
\end{equation}
where $(H_1R)_G\simeq R/[R,F]$ and $H_1F\simeq F/[F,F]$, therefore, we obtain an analogue of the Schur--Hopf formula provided that $H_2F=0$.
Clearly this is the case for a free group $F$, still we can consider a broader variety of presentations.
In this regard our main result is the following (the join of Theorem \ref{Satz:Hopf properties of the covers} and \ref{Satz:hyperiodic covers have minimal exponent}).
\begin{theorem}\label{Satz:main}
 Let $1\to R\to F\to G\to 1$ be a presentation of a finite group $G$ by a finite free product of finite cyclic groups $F$, then $E=F/[R,F]$ is a finite cover of $G$.
 Moreover, if the projection $F\to G$ preserves the order of the generators, then $E$ has minimal exponent among the covers.
\end{theorem}
The first statement is elementary and already it simplifies the classical situation: here the group $E$ is finite and does not require any choice to be constructed.
On the other hand, the order preserving property relates to the unitary cover, which has been introduced by the author to describe bounds for the exponent of the multiplier  \cite{Sambonet2015,Sambonet2017}.
The unitary cover has minimal exponent among the covers, and there are basic examples where the minimal exponent is not realized by any Schur cover.
Moreover, the unitary cover is defined by means of an identity and not by a choice.
In fact a formula for the unitary cover exists (Theorem \ref{Satz:the unitary cover is universal}) and this is the major step towards the proof of Theorem \ref{Satz:main}, which depends on the general property that the Schur construction is natural with respect to subgroups and generation (Lemma \ref{Satz:Schur construction and generators}).
Another remarkable fact is that the new shape of the formula yields a notion of growth which is absent in the classical theory, but is inline with other contemporary topics such as the Burnside problem and coclass theory, which are central in the current discussion of the exponent problem \cite{MoravecCoclass,Sambonet2017,VaughanLee2021}.
Thus any presentation as in Theorem \ref{Satz:main} affords a \emph{profinite periodic cover} $E_\infty=\varprojlim F/[R,_kF]$, that is a profinite group obtained as an inverse limit of a sequence of covers, and it is a pro-$p$ group provided that both $G$ and $F/[F,F]$ are $p$-groups.
Based on a famous theorem of K.~Iwasawa, we prove the following result (Theorem \ref{Satz:profinite covers})  which extends the characterization of the $p$-groups having trivial multiplier due to D.~L.~Johnson \cite{Iwasawa,Johnson}.
\begin{theorem}
 If a finite group $G$ has non cyclic abelianization, then all of its profinite periodic covers have infinite order.
\end{theorem}
While the first part of this manuscript is devoted to purely algebraic aspects of the Schur--Hopf formula, the second part focuses on the underlying topology.
For Lie groups and for topological groups that are connected, locally connected, and locally simply connected, the covering spaces have a natural group structure and so they are the covering groups \cite{Pontryagin}.
However, finite groups have discrete topology, in this respect any surjective map between them is a covering projection.
The use of the same terminology in representation theory is explained by the fact that often a Schur cover of a perfect group of Lie type corresponds to the topological covering group of its complex analogue, an example is given by $\SL_n(\F)$ and $\PSL_n(\F)$.
Noteworthy, we establish a new relation between covering groups and covering spaces by considering the cellular and simplicial complexes properly related to Theorem \ref{Satz:main}.
In fact, in the classical proof of the formula, that is when $F$ is free, we identify $G$ with the group of deck transformations associated to some covering projection  $p:Y\to X$.
Here $Y$ is the Cayley graph associated with the presentation, and $X$ is a bouquet of circumferences, whose universal cover is a tree $\tilde X$.
Therefore, $\pi_1X\simeq F$, $\pi_1 Y\simeq R$ and $\Aut(Y/X)\simeq G$ \cite{Brown}.
In the situation of Theorem \ref{Satz:main}, we have a free product 
\[F=\gen{\ f_1,\ldots,f_d\ |\ f_1^{m_1},\ldots, f_d^{m_d}\ }\]
and so we attach 2-cells corresponding to the relations $f_i^{m_i}=1$ for $i=1,\ldots,d$ to the above graph and obtain an analogue for $Y$, which is a two dimensional complex.
On the other hand, also the cover $E=F/[F,R]$ is presented by $F$, and so we associates a complex $W$ to $E$.
In this way, we obtain a new elegant correspondence between an algebraic and a geometric notion (Corollary \ref{Satz:hyperiodic and local homeomorphism}).
\begin{corollary}\label{Satz:corollary in topology}
The group homomorphism $\pi:E\to G$ induces a continuous map $p:W\to Y$ between cellular complexes.
Such $p$ is a covering projection precisely when $\pi$ preserves the order of the generators, and in this case $E$ has minimal exponent among the covers of $G$.
\end{corollary}
On the other hand, covering spaces were first introduced to study Riemann surfaces by F.~Klein, L.~Fuchs, A.~M\"obius, and H.~Poincar\'e, where they permit a deep use of group theory in analogy with Galois theory for fields extension  \cite{AhlforsSario}.
Compact Riemann surfaces have finite group of isometries, and representation theory continues to produce important results, for instance with the work of G.A.~Jones, M.W.~Liebeck and A.~Shalev \cite{Jones,Liebeck}.
Going in the other direction, T.~Breuer has studied the representations arising from finite groups acting on Riemann surfaces \cite{Breuer}.
 In addition, the Kleinian and Fuchsian groups, which are discrete subgroups of $\PSL_2(\C)$ classifying the Riemann surfaces,  can be studied by purely combinatorial means, and forgetting the complex analytic structure provides discrete geometries such as regular tilings and polytopes, inline with the nature of homology of finite groups which is based on simplicial complexes \cite{Conway,Coxeter,Magnus}.
 The homology of Fuchsian groups, and generalized triangle groups, has been computed by S.~J.~Patterson, G.~Ellis and G.~Williams \cite{EllisWilliams,Patterson}, moreover, T.W.Tucker has studied finite groups acting on combinatorial surfaces in close analogy with the conformal actions \cite{Tucker}, and it is worth to say that the Schur multiplier has important applications in combinatorial group theory \cite{Baumslag,LyndonSchupp}.
 In this respect, by imposing a further relation to the above free products, we obtain a Fuchsian group generated by elliptic elements
\[\Delta=\gen{\ y_1,\ldots,y_{d+1}\ |\ y_1^{m_1}\ ,\ \ldots\ ,\ y_{d+1}^{m_{d+1}}\ ,\ y_1\cdots y_d\ }\]
where $m_1,\ldots,m_{d+1}$ are integers greater than $1$.
Therefore, we consider an \emph{elliptic presentation} of a finite group $G$, that is a sequence $1\to S\to\Delta\to G\to 1$ satisfying the condition that $o(\pi y_i)=m_i$ for all $i=1,\ldots,d+1$, and this affords a finite central extension $D=\Delta/[S,\Delta]$ of $G$.
There is an obstruction for $D$ being a cover, since $H_2\Delta$ is non necessarily trivial and so the first term in \eqref{equation:LHS short exact sequence} does not vanish.
Nonetheless, we prove that an elliptic cover always exists and, interestingly, the relevant topological spaces are now simplicial oriented compact surfaces.
\begin{theorem}
For any elliptic presentation of a finite group, the homomorphism  $\Delta/[S,\Delta]\to G$ induces a covering projection between compact oriented surfaces.
 Moreover, an elliptic cover always exists.
\end{theorem}
We prove this result in a more general form (the join of Theorem \ref{Satz:existence of elliptic covers} and \ref{Satz:elliptic covers and topology}), in fact we describe the simplicial structure underlying the compact Riemann surfaces, and we obtain some insights for classical results in terms of discrete geometry and combinatorial group theory.
The manuscript is structured as follows.
In section \ref{section:background}, we recall some elementary aspects of Schur's theory on projective representations and the unitary cover, referring to \cite{Isaacs,Sambonet2015,Sambonet2017}.
In section \ref{section:Hopf finite cyclic}, we study the Schur--Hopf formula with respect to free products of finite cyclic groups, and in section \ref{section:topology} we address the case of Fuchsian groups generated by elliptic elements together with the topology underlying both situations.

\section{Background}\label{section:background}
The first way to study a projective representation $\varphi:G\to\PGL_n(\C)$ is to associate it with an element of the \emph{Schur multiplier} as follows.
A \emph{section} for $\varphi$ is a map $\tau:G\to\GL_n(\C)$ making the diagram
\[\xymatrix{
&G\ar[d]^\varphi\ar[dl]_{\tau}\\
\GL_n(\C)\ar[r]&\PGL_n(\C)
}\]
commutative.
The failure for $\tau$ being a homomorphism is encoded in a function $\alpha:G\times G\to\C^\times$ defined by the equality
$\tau(g)\tau(h)=\tau(gh)\alpha(g,h)$.
Clearly not any function can be found in this way, for instance, by computing $\tau(x)\tau(y)\tau(z)$, the associativity of $\GL_n(\C)$ implies that $\alpha$ satisfies the identity \[\alpha(x,y)\alpha(xy,z)=\alpha(x,yz)\alpha(y,z)\ .\]
The functions satisfying this identity are the \emph{cocycles}, and they constitute a group $Z^2G=Z^2(G,\C^\times)$ under point-wise multiplication.
It can be proved that every cocycle arises from some projective representation.
On the other hand, by changing the section $\tau$ for another $\tau'$, we must have $\tau'(g)=\tau(g)\zeta(g)$ for some function $\zeta:G\to\C^\times$.
In $Z^2G$ this corresponds to multiplication of $\alpha$ by the \emph{coboundary} $\delta\zeta$, which is defined by \[\delta\zeta(g,h)=\zeta(g)\zeta(h)\inv{\zeta(gh)}\ .\]
To forget the choice of section we factor the subgroup $B^2G=B^2(G,\C^\times)$ of coboundaries.
Therefore,  the fundamental invariant that classifies the projective representations of $G$ is $H^2(G,\C^\times)=Z^2(G,\C^\times)/B^2(G,\C^\times)$, this is the \emph{Schur multiplier} which we denote $H^2 G$.

The second way to study a projective representation is to solve the \emph{lifting problem} for $\varphi$, namely, to determine a finite extension $1\to A\to E\to G\to 1$, together with an ordinary representation $\hat\varphi$ which makes the diagram
\[\xymatrix{
E\ar[r]^\pi\ar[d]_{\hat\varphi}&G\ar[d]^\varphi\\
\GL_n(\C)\ar[r]&\PGL_n(\C)
}\]
commutative.
Here the homomorphism $\pi$ is surjective and, since $\PGL_n(\C)$ is the quotient of $\GL_n(\C)$ over its center $\C^\times$,
there is no loss by assuming that $A=\ker\pi$ is central in $E$, that is $[E,A]=1$.
As above, a section $\sigma:E\ot G$ for $\pi$ determines a cocycle $\gamma$ in $Z^2(G,A)$, and thus $E$ determines a coclass $[\gamma]$ in $H^2(G,A)$.
In this respect, there is a homomorphism
\begin{equation}\label{eq:the standard map}
 \eta:\check{A}=\Hom(A,\C^\times) \to H^2G\ \ ,\ \ \eta(\lambda)=[\lambda\circ\gamma]
\end{equation}
the \emph{standard map}, which address the lifting problem: the lifting $\hat\varphi$ exists if and only if the coclass $[\alpha]$ afforded by $\varphi$ belongs to the image of $\eta$.
In turn
\[\eta(\check{A})\simeq[E,E]\cap A\ .\]
We observe that this ammounts to prove that $\ker\eta=([E,E]\cap A)^\perp$ since, for a generic subgroup $B$ of $A$, there is a natural duality isomomorphism $\check B\simeq\check{A}/B^\perp$ where $B^\perp=\{\lambda\in A\ |\ B\leq\ker\lambda\}$ (see \cite{Isaacs}).
This gives a complete picture of the algebraic notion of cover.
\begin{definition}
A \emph{cover} of a finite group $G$ is a finite central extension \[1\to A\to E\to G\to 1\ \ ,\ \ [E,A]=1\]
satisfying the following equivalent conditions:
\begin{enumerate}[i)]
 \item every projective representation of $G$ can be lifted to $E$
 \item the standard map $\eta$ is surjective
 \item the subgroup $[E,E]\cap A$ is isomorphic to $H^2G$
\end{enumerate}
If the standard map $\eta$ is an isomorphism then $E$ is a \emph{Schur cover}, in this case $A$ is a subgroup of $[E,E]$ isomorphic to $H^2G$.
\end{definition}
There is a way somehow to invert the above process.
The \emph{Schur construction}
associates any given finite subgroup $S$ of $Z^2G$ to a finite central extension 
\[1\to\check S\to\check S\propto G\to G\to 1\ ,\]
in such a way that the projective representations of $G$ which can be lifted to $\check S\propto G$ are those providing a coclass represented in $S$
(we provide few more details at the end of this section).
In particular, having that $B^2G$ is a divisible finite index subgroup of $Z^2G$ and as such it is complemented, we can write $Z^2G=B^2G\oplus J$ to obtain a
the Schur cover $\check J\propto G$.
Thus the lifting problem always admits a positive solution, and this is the fundamental theorem of the whole theory unifying projective and ordinary representations.
\begin{theorem}[Schur 1904]
\label{Satz:fundamental theorem of projective representations}
 Any finite group admits a Schur cover.
\end{theorem}
By definition, the Schur covers have minimal order among the covers.
The fact that a group possibly has many non-isomorphic Schur covers follows by the choice of the complement $J$.
With this theorem, the formula can be proved using the universal property of $F$, once we fix a cover $E$, the homomorphism $E\to G$ lifts to a homomorphism $F\to E$, mapping $[F,F]\cap R$ onto $[E,E]\cap A\simeq H^2G$.
\begin{theorem}[Schur 1907]\label{Satz:Schur formula}
 Any presentation $1\to R\to F\to G\to 1$ of a finite group $G$ by a free group $F$ provides a isomorphism \[H^2G\simeq [F,F]\cap R/[R,F]\ .\]
 Moreover, $R/[F,F]\cap R$ is a free abelian group.
\end{theorem}
In order to study the exponent of the multiplier $H^2G$, we may look at how the powers behave in some cover of $G$. 
More in general, we observe that in any central extension $1\to A\to E\to G\to 1$,
if the section $\sigma:E\ot G$ afford the cocycle $\gamma$ we have
\[\sigma(g)^{o(g)}=\prod_{j=0}^{o(g)-1}\gamma(g,g^j)\]
for all element $g$ in the group $G$.
We come to the following notion \cite{Sambonet2015}.
\begin{definition}[2015]
 The \emph{unitary cover} of a finite group $G$ is \[\GU G=(\ZU G)\check\ \propto G\]
 where $\ZU G$ denotes the group of the cocycles satisfying the identity
 \[\prod_{j=0}^{o(g)-1}\alpha(g,g^j)=1\ \ ,\ \ \forall g\in G\]
 and are called the \emph{unitary cocycles}.
\end{definition}
In fact, since $\ZU G$ is finite the Schur construction is well defined, and since $\ZU G$ represents the whole multiplier the result is a cover.
Moreover, we see that $\GU G$ is defined by an identity and not by a choice.
Remarkably:
\begin{theorem}[2015]\label{Satz:the unitary cover has minimal exponent}
The unitary cover has minimal exponent.
\end{theorem}
In addition, the unitary condition is preserved by restriction to subgroups and inflation from quotients and in turn the unitary cover presents a functorial behavior.
This fact provides a tool to address the exponent problem by having that
\[\exp\GU G\ |\ \exp\GU(N)\cdot\exp\GU(G/N)\]
for any normal subgroup $N$ of $G$.
Moreover, we have that
\[\exp\GU G=\lcm\{\exp G,\ \exp\ZU G\}\]
this allows computation by looking at each value $\alpha(g,h)$ separately, and an interesting consequence is that
\begin{equation}\label{eq:unitary cover and two generated subgroups}
 \exp\GU G=\lcm\{\exp\GU(\gen{g,h})\ |\ g,h\in G\}
\end{equation}
and the exponent problem is controlled by the two generated subgroups \cite{Sambonet2017}.

It is time for us to describe the Schur construction with some more details.
The \emph{Schur construction} associates a given finite subgroup $S$ of $Z^2(G)$ to the finite central extension
$1\to\check S\to\check S\propto G\to G\to 1$
whose underling set is $G\times \check S$, and multiplication is given by the rule \[(g,\vartheta)\cdot(h,\psi)=(gh,\omega(g,h)\vartheta\psi)\] where
\begin{equation}\label{eq:multiplication cocycle in the Schur construction}
 \omega(g,h)=\{\alpha\mapsto\alpha(g,h)\}\in\check S=\Hom(S,\C^\times)\ .
\end{equation}
This relates to the standard map as follows.
Reading \eqref{eq:the standard map}, we see that $\eta$ is the composite of the homomorphism
\begin{equation}\label{eq:the standard map at the level of cocycles}
 \dot\eta:\check{A}\to Z^2G\ \ ,\ \ \dot\eta(\lambda)=\lambda\circ\gamma
\end{equation}
with the natural projection from $Z^2G$ to $H^2G$, here $\gamma$ denotes the cocycle associated to a section $\sigma:E\ot G$.
We observe that with respect to $\check{S}\propto G$ and the canonical section $\upsilon:(g,1_S)\mapsfrom g$, the map $\dot\eta$ is the natural isomorphism $\checkk{S}\simeq S$, and this fact illustrate our previous claim that the image of $\eta$ are precisely the coclasses represented in $S$.
On the other hand, for a finite central extension $1\to A\to E\to G\to 1$, with section $\sigma$, cocycle $\gamma$, and associated map $\dot\eta$, we may apply the Schur construction to $
\dot\eta(A\check\ )$ to get a new extension.
The cocycle $\omega$ defining the multiplication in $\dot\eta(A\check\ )\check\ \propto G$ is related to $\gamma$ via the formula
\begin{equation}\label{eq:duality standard}
 \omega(g,h)(\dot\eta(\lambda))=\lambda(\gamma(g,h))
\end{equation} for all $g$ and $h$ in $G$, and $\lambda$ in $\check{A}$.
We can now establish some natural properties of the Schur construction.
\begin{lemma}\label{Satz:Schur construction and generators}
Let $G$ be a finite group, $S$ and $T$ be finite subgroups of $Z^2G$ with $T\leq S$, and $1\to A\to E\to G\to 1$ be a finite central extension. Then
\begin{enumerate}[i)]
 \item $\check S\propto G=\gen{\upsilon{(G)}}$, for the canonical section $\upsilon:(g,1_S)\mapsfrom g$.
 \item $\check T\propto G\simeq(\check S\propto G)/T^\perp$, where $T^\perp=\{(1,\lambda)\ |\ \lambda\in\check S\ ,\ T\leq\ker\lambda\}$.
 \item $\gen{\sigma(G)}\simeq \dot\eta(A\check\ )\check\ \propto G$ for any section $\sigma:E\ot G$ affording $\dot\eta:\check{A}\to Z^2G$.
\end{enumerate}
Moreover, the above isomorphism are natural.
\end{lemma}
\begin{proof}
i)
For any subset $\Lambda$ of $\check{A}$ we have $\gen{\Lambda}=K^\perp$ where $K=\bigcap_{\lambda\in\Lambda}\ker\lambda$, consequently $\Lambda$ generates $\check{A}$ if and only if $K=1$.
We observe that
\[(1,\omega(g,h))=(1,1_S)\cdot(g,1_S)\cdot(h,1_S)\cdot\inv{(gh,1_S)}\in\gen{\upsilon(G)}\] for all $g$ and $h$, since $\ker\omega(g,h)=\{\alpha\in S\ |\ \alpha(g,h)=1\}$ then
\[\bigcap\left\{\ker\omega(g,h)\ |\ (g,h)\in G\times G\right\}=1\]
so that $\check S=\gen{\ \omega(g,h)\ |\ (g,h)\in G\times G\ }$.
ii) follows immediately by i) and the duality isomorphism $\check{S}/T^\perp\simeq\check{T}$.
iii) We know by i) that the generic element of $\dot\eta(A\check\ )\check\ \propto G$ can be written as
$(g,\prod_{i}\omega(g_i,h_i))$ for suitable $g_i,h_i$ and $g$ in $G$.
On the other hand, the generic element of $\gen{\sigma(G)}$ is $\sigma(g)\cdot\prod_{i}\gamma(g_i,h_i)$, so we define \[\varphi:\sigma(g)\cdot\prod_{i}\gamma(g_i,h_i)\mapsto(g,\prod_{i}\omega(g_i,h_i))\ .\]
Observe that the map $\varphi$ is well defined
by \eqref{eq:duality standard}, and it is enough to check the definition of the product in $E$ and $\check{S}\propto G$ to see that $\varphi$ is an isomorphism.
\end{proof}

\section{Periodic and hyperiodic covers}\label{section:Hopf finite cyclic}
We write the Hopf formula in one of its generalized versions, allowing a presentation by a free product of arbitrary cyclic groups.
So we let
\[F=\Ast_{i=1}^d\Z_{m_i}\ \ ,\ \ \Z_{m_i}=\gen{f_i}\ ,\]
and call the order of the free generators $m_i=o(f_i)$ the \emph{periods} of $F$.
By a standard Mayer-Vietoris argument we have that $H_2F=0$, by reading \eqref{equation:LHS short exact sequence} we obtain the desired version of the formula.
\begin{theorem}[Periodic Schur--Hopf formula]\label{Satz:Hopf formula}
Let $1\to R\to F\to G\to 1$ be a group presentation by a free product of cyclic groups.
Then
\[H_2G\simeq[F,F]\cap R/[R,F]\ .\]
\end{theorem}
In view of this formula, it is natural to consider the group $E=F/[R,F]$ together with the filtration $[R,F]\leq[F,F]\cap R\leq R\leq F$ having $G$ and $H_2G$ among its factors.
If some of the periods are infinite, just as in the classical situation, then $F/[F,F]$ is an infinite group and at the more so $E$ is such.
Therefore, focusing on finite groups it is of interest to consider the case in which all of the periods are finite.
\begin{definition}
 A group presentation $1\to R\to F\to G\to 1$ where $G$ is finite and $F$ is a free product of finite cyclic groups is a \emph{periodic presentation}, and the group $E=F/[R,F]$ is the \emph{periodic cover} afforded by the presentation.
\end{definition}
The fundamental fact is that the periodic covers are precisely the finite covers which arise from presentations by a free product of cyclic groups.
\begin{theorem}\label{Satz:Hopf properties of the covers}
Let $1\to R\to F\to G\to 1$ be a group presentation by a free product of cyclic groups.
Then $E=F/[R,F]$ is a finite cover of $G$ if and only if the presentation is periodic, in which case 
$|E|=|F:[F,F]|\cdot |[G,G]|\cdot|H_2G|$.
Moreover, $E$ is a Schur cover of $G$ if and only if $R\leq[F,F]$, and $E$ is a $p$-group if and only if $G$ is such and all the periods are $p$-powers.
\end{theorem}
\begin{proof}\label{Satz:covers prop}
We consider the filtration $[R,F]\leq[F,F]\cap R\leq R\leq F$ and use the Noether isomorphism $R/[F,F]\cap R\simeq [F,F]R/[F,F]$. So
\[|F:[F,F]|=|F:[F,F]R|\cdot|[F,F]R:[F,F]|=|G:[G,G]|\cdot|R:[F,F]\cap R|\ .\]
The condition $R\leq[F,F]$ is equivalent to $[F,F]\cap R=R$, whence the above filtration reduces to $[R,F]\leq R\leq F$ and $E$ is a Schur cover.
If $G$ is a $p$-group it is known that the multiplier is also a $p$-group, so that $|[G,G]|\cdot|H_2 G|$ is a $p$-power.
On the other hand $|F/[F,F]|=\prod_im_i$, and the result follows.
\end{proof}
Dealing with finite groups allows representation theory.
Moreover, in the category of finite groups with a fixed set of generators, the new formula carries a natural notion of growth.
Among the periodic presentations some deserve a major attention.
\begin{definition}
 A periodic presentation is \emph{hyperiodic} if the group homomorphism which maps $F$ onto $G$ preserves the order of the generators.
 In this case, we say that $E=F/[R,F]$ is a \emph{hyperiodic cover}.
\end{definition}
 We shall see that $\GU G$ is hyperiodic, and that all the hyperiodic covers have minimal exponent. Later we will show that they corresponds precisely to the having a local-homeomorphism between the relevant topological spaces.
 To begin with we establish the existence of a Hopf formula for the unitary cover.
 To this aim we consider the periodic presentation
\[1\to R_u\to F_u\to G\to 1\ \ ,\ \ F_u=\Ast_{g\in G}\Z_{o(g)}\ ,\]
where the periods are precisely the orders of the group elements and for each free factor $\Z_{o(g)}$ the cyclic generators $f_g$ is mapped to $g$.
In turn, the unitary cover $\GU G$ is naturally isomorphic with the hyperiodic cover $E_u=F_u/[R_u,F_u]$.
\begin{theorem}\label{Satz:the unitary cover is universal}
 The covers $\GU G$ and $E_u$ are naturally isomorphic, by pairing the generators $(g,1_u)$ and $R_uf_g$ for all $g$ in $G$. 
\end{theorem}
\begin{proof}
We prove that both $E_u$ and $\GU G$ share the following universal property.
We consider pairs $(\Gamma,\sigma)$ consisting of a finite central extension which admits an order preserving section $\sigma:\Gamma\ot G$, and given two such $(\Gamma,\sigma)$ and $(\Gamma',\sigma')$ we look at the group homomorphisms $\varphi:\Gamma\to \Gamma'$ which respects the sections, that is $\sigma'=\sigma\circ\varphi$.
A pair $(U,\nu)$ is \emph{universal} if it maps uniquely
$(U,\nu)\to(\Gamma,\sigma)$ over each pair. As usual, given two universal pairs $(U,\nu)$ and $(U',\nu')$, the unique homomorphisms $\varphi:U\to U'$ and $\varphi':U'\to U$ are inverse isomorphisms, so that the isomorphism class of the underling group of an universal pair is uniquely determined.
First, we consider the pair $(E_u,\epsilon)$, where $\epsilon:f_g\mapsfrom g$.
Given $(\Gamma,\sigma)$ as above, since $\sigma$ is order preserving, we have an isomorphism $\gen{f_g}\to \gen{\sigma(g)}$, $f_g\mapsto\sigma(g)$, for each $g$ in $G$.
By the universal property of free products, we obtain a unique homomorphism $\varphi:E_u\to\Gamma$ such that $\varphi(f_g)=\sigma(g)$ and, since $f_g=\epsilon(g)$, we have $\varphi\circ\epsilon=\sigma$.
This proves the universality of $(E_u,\epsilon)$.
Then, we consider the pair $(\GU G,\upsilon)$, where $\upsilon:(g,1_u)\mapsfrom g$.
 Given a pair $(\Gamma,\sigma)$ as above, there is no loss of generality by assuming that $\Gamma=\gen{\sigma(G)}$.
 We write $\pi:\Gamma\to G$ and $A=\ker\pi$, and we consider the map $\dot\eta:\check{A}\to Z^2G$ associated with $\sigma$, as in \eqref{eq:the standard map at the level of cocycles}.
 Since $\sigma$ is order preserving, it follows that $\dot\eta(\check{A})$ is contained in $\ZU G$.
 Applying Lemma \ref{Satz:Schur construction and generators}, respectively the statements $iii)$ and $ii)$, it follows that $\Gamma$ is isomorphic to $\dot\eta(A\check\ )\check\ \propto G$, and that the latter is a homomorphic image $\GU G$.
 By checking the definitions, it follows that the composite homomorphism $\varphi:\GU G\to \Gamma$ is such that $(g,1_u)\mapsto\sigma(g)$, in particular $\varphi$ is unique and $\sigma=\upsilon\circ\varphi$, therefore $\GU G$ is universal.
\end{proof}
 By combining the theorems \ref{Satz:the unitary cover has minimal exponent} and \ref{Satz:the unitary cover is universal} we see that $E_u$ has minimal exponent among the covers. 
 For a hyperiodic cover $E$ Some caution has to be paid since it is not granted the existence of an order preserving section. A basic example of this comes with the dihedral group $D_8$ which is a hyperiodic cover of $\Z_2\times\Z_2$, since it is generated by two involutions, but it does not admit an order preserving section, since the product of the two involutions has order 4.
 Nonetheless, the property of having minimal exponent  is shared by all the hyperiodic covers.
 \begin{theorem}
 \label{Satz:hyperiodic covers have minimal exponent}
 Hyperiodic covers have minimal exponent.
\end{theorem}
\begin{proof}
 Given a periodic presentation $1\to R\to F\to G\to 1$ together with a cyclic group $\gen{y}$, we write $F'=F\ast\gen{y}$, and we extend to $F'=F\ast\gen{y}$ the homomorphism $\pi:F\to G$ by assigning the image of $y$.
 Thus $\pi':F'\to G$ satisfy $\pi'\mid_F=\pi$, and let $\pi'(y)=g$.
 Since $\pi$ is surjective, $g=\pi(w)$ for some $w$ in $F$.
 Then $z=y\inv{w}$ lies in $R'=\ker\pi'$, so that $\bar z=[R',F']z$ is central in $E'=F'/[R',F']$.
 Moreover, by induction on $k$, we have that \[z^{k+1}=(y^{k}w^{-k})z\equiv y^{k}zw^{-k}=y^{k+1}w^{-(k+1)}\mod[R',F']\ .\]
 In particular, we have $\bar z^{o(g)}=\bar w^{-o(g)}$ and $o(\bar z)=o(\bar w)$ under the additional condition $o(y)=o(g)$.
 Since $E'=\gen{E,y}=\gen{E,z}$, it follows that the cover $E'$ is a central product of $E$ and $\gen{z}$, with amalgamation over $C=\gen{\bar z^{o(g)}}\simeq \gen{\bar w^{-o(g)}}$.
 Since the order of $\bar z$ coincides with the order of the element $\bar w$ of $E$, it follows that $\exp E'=\exp E$.
 Now, given a hyperiodic presentation, we observe that all the generators satisfy the condition $o(y)=o(g)$ necessary to use the above fact.
 Therefore, it is possible to reach
  $F'=F\ast F_u$ by adding cyclic free factors either from $F$ or from $F_u$, and this procedure does not increase the exponent. Thus $\exp E=\exp E'=\exp E_u$, which is minimal.
\end{proof}
Is a periodic cover of a given group a proper extension?
This natural question introduces a notion of growth moving the focus to profinite groups, thus, given a periodic presentation $1\to R\to F\to G\to 1$ of a finite group, we consider the inverse system of periodic covers
 \[\cdots\to F/[R,_{k+1}F]\to F/[R,_kF]\to \cdots\to F/[R,F]\to G\simeq F/R\]
 where, as usual, $[R,_1F]=[R,F]$ and $[R,_{k+1}F]=[[R,_k F],F]$ for $k\geq 1$.
 \begin{definition}
 For a periodic presentation $1\to R\to F\to G\to 1$ of a finite group $G$, we say that $E_\infty=\varprojlim F/[R,_kF]$
 is a \emph{profinite periodic cover} of $G$.
 \end{definition}
The above question is refined by asking: is $E_\infty$ an infinite group?
The answer to both questions comes with Iwasawa's theorem claiming that any free group is a residually finite $p$-group, for any prime $p$.
It follows that any free power $(\Z_p)^{\ast d}$ is residually nilpotent (in fact this can be read through the proof of the theorem, see \cite{Robinson}), and this allows us to show that finite groups with non cyclic abelianization have non trivial periodic covers.
\begin{theorem}\label{Satz:profinite covers}
 If a finite group has non cyclic abelianization, then all of its periodic covers are proper, and all of its profinite periodic covers are infinite groups.
\end{theorem}
\begin{proof}
First we show that, given a free product of cyclic groups $F$ which maps onto a non cyclic elementary abelian $p$-group with kernel $K$, then $[K,_{k+1}F]$ is properly contained in $[K,_kF]$ for any $k\geq 0$.
To this aim, denote $f_1,\ldots,f_d$ the generators of the free factors of $F$, and let $T=\gen{\{f_i\ |\ f_i\in K\},\{f_j^p\ |\ f_j\notin K\}}^F$.
 We have $K=\gamma_2(F)T$, and more generally $[K,_kF]T=\gamma_{k+2}(F)T$ for every $k\geq 0$.
 This proves the claim as $F/T$ is isomorphic with a free product of cyclic groups of order $p$, as such it is residually nilpotent by Iwasawa's theorem, and it is an infinite group since its abelianization is non cyclic.
 Now we are ready to prove the theorem.
 By hypothesis $G$ has non cyclic abelianization, so it maps onto some non cyclic elementary abelian $p$-group $A$.
 Let $1\to R\to F\to G\to 1$ be a periodic presentation.
 By composition we get a homomorphism from $F$ onto $A$, and we denote by $K$ its kernel.
 Clearly $R$ is contained in $K$, and as we first proved there exists $l\geq 0$ such that $R\leq [K,_l F]$ and $R\nleq[K,_{l+1}F]$.
 This shows that $[R,F]$ is properly contained in $R$.
\end{proof}
 In particular, any periodic cover of a non cyclic $p$-group $G$ is proper: this statement is equivalent to Johnson's characterization of the non cyclic $p$-groups with non trivial multiplier \cite{Johnson}.
 Moreover, it follows by Theorem \ref{Satz:Hopf properties of the covers} that, for any $p$-group $G$, a $p$-periodic presentation (i.e.~such that the periods are $p$-powers) affords an infinite pro-$p$ group.
 Moreover, under the order preserving condition, a hyperiodic presentation yields an inverse system of hyperiodic covers.
 Another remarkable fact, which also follows by Theorem \ref{Satz:Hopf properties of the covers}, is that after the first step any inverse system continues with hyperiodic Schur covers.
\begin{corollary}
In any profinite periodic cover of a finite group, the group $F/[R,_{k+1}F]$ is a hyperiodic Schur cover of $F/[R,_kF]$ for any $k\geq 1$.
\end{corollary}
It is worthy to mention a result of N.~Iwahori and H.~Matsumoto, stating that for any Schur cover $Y$ of a group $X$, then $H_2Y$ embeds into $X_{ab}\otimes H_2G$ \cite{IwahoriMatsumoto}.
We see that the growth rate of a profinite periodic cover is somehow controlled by the first two terms, the group $G$ and its periodic cover $E=F/[R,F]$.

It is interesting to compare this result with the affirmative solution of the restricted Burnside problem:
\emph{there are only a finite number of finite groups with $d$ generators and exponent $e$} \cite{VaughanLee}.
The two key ingredients in the proof are the Hall--Higman reduction of the problem to the case of $p$-groups,
and the Lie algebras technique of A.~I.~Konstrikin and E.~Zelmanov.
Now, given a periodic presentation $1\to R\to F\to G\to 1$, for any multiple $e$ of the exponent of $G$, one has an extension $F/[R,F]F^e$, and thus an inverse system
\[\cdots\to F/[R,_{k+1}F]F^e\to F/[R,_kF]F^e\to \cdots\to F/[R,F]F^e\to G\simeq F/R\ .\]
Assuming the Hall--Higman reduction, since $p$-groups are nilpotent, the solution of the restricted Burnside problem can be rephrased as follows: \emph{the above inverse system is stationary for any periodic presentation of a finite group.}

\section{Elliptic covers and topology}\label{section:topology}
In view of \eqref{eq:unitary cover and two generated subgroups}, to determine the exponent of the unitary cover $\GU G$, we can look at each two-generated subgroup separately, and calculate the order of the product $\upsilon(g_1)\upsilon(g_2)$ in $\GU(\gen{g_1,g_2})$, where $\upsilon$ denotes the canonical section.
This suggests to look at the hyperiodic presentations of $\gen{g_1,g_2}$ by $F=\gen{f_1}\ast\gen{f_2}$, where $f_i\mapsto g_i$, and we let $f_3=f_1f_2\mapsto g_3=g_1g_2$, the order of $[R,F]f_3$ carries some information about the exponent of $\GU G$.
Therefore, we may consider the presentation by the triangle group $\Delta=\gen{\ y_1,y_2,y_3\ |\ y_i^{m_i}\ ,\ y_1y_2y_3\ }$, where $y_i\mapsto g_i$
and $m_i=o(g_i)$, with kernel $S$.
The above information is encoded in the kernel $Z$ of the cyclic extension $1\to Z\to E\to D\to 1$, where $E=F/[R,F]$ and $D=\Delta/[S,\Delta]$,
This provides the algebraic motivation to consider the following presentations, where we allow any number of generators.

We consider the finitely presented group
\[\Delta=\gen{\ y_1,\ldots,y_{d+1}\ |\ y_1^{m_1}\ ,\ \ldots\ ,\  y_{d+1}^{m_{d+1}}\ ,\ y_1y_2\cdots y_{d+1}\ }\]
where $m_1,\ldots,m_d$ are integers greater than 1.
The group $\Delta$ is a proper Fuchsian group generated by elliptic elements \cite{Liebeck}.
We call $m_\ast=(m_1,\ldots,m_{d+1})$ the \emph{signature} of $\Delta$
and write $\Delta=\Delta(m_\ast)$.

In general, for a finitely generated group $G$, we call a \emph{generating system} any ordered sequence $g_\ast=(g_1,\ldots,g_d)$ such that $G=\gen{g_1,\ldots,g_d}$, and we say that $g_\ast$ is \emph{periodic} if each period $m_i=o(g_i)$ is an integer greater than 1.
Thus, an \emph{elliptic generating system}  is a periodic generating system $g_\ast=(g_1,\ldots,g_{d+1})$ such that $g_1g_2\cdots g_{d+1}=1$, and its \emph{signature} is the sequence $m_\ast=(m_1,\ldots,m_{d+1})$.
Our interest is set on the following definition.
\begin{definition}
 Let $G$ be a group admitting an elliptic generating system $g_\ast$ of signature $m_\ast$, and let $\Delta=\Delta(m_\ast)$.
 Then the group extension
\[1\to S\to\Delta\to G\to 1\ \ \ ,\ \ \  y_i\mapsto g_i\]
is an \emph{elliptic presentation} of $G$, and
\[1\to S/[S,\Delta]\to\Delta/[S,\Delta]\to G\to 1\ \ ,\ \ D=\Delta/[S,\Delta]\]
is an \emph{elliptic central extension} of $G$.
\end{definition}
Any hyperiodic presentation of a finite group $G$ affords in a natural way an elliptic presentation.
First, if the product $f_1\cdots f_d$ does not belong to the relators $R$, we add a cyclic free factor $\gen{f_{d+1}}$ of order $m_{d+1}$ to the group $F$, and extend the homomorphism $\pi:F\to G$ by setting $\pi(f_{d+1})=g_{d+1}$, otherwise we just write $d+1$ for $d$.
Therefore, we may assume that
\begin{equation}\label{eq:hyperiodic presentation}
 1\to R\to F\to G\to 1\ \ ,\ \ F=\Ast_{i=1}^{d+1}\gen{f_i}\ \ ,\ \ f_i\mapsto g_i\ ,\ 
\end{equation}
where $g_1\cdots g_{d+1}=1$, and $m_i=o(g_i)=o(f_i)$ is an integer greater than 1 for all $i=1,\ldots,d+1$. Thus $\Delta=\Delta(m_\ast)$ is hyperiodically presented by $F$ as
\begin{equation}\label{eq:presenting Delta}
1\to T\to F\to \Delta\to 1\ \ \ ,\ \ \  T=\gen{f_{1}\cdots f_{d+1}}^F
\end{equation}
On the other hand, since $f_1\cdots f_{d+1}\in R$, then $\pi:F\to G$ factors through $\pi':F\to \Delta$, providing an elliptic presentation of $G$.
Moreover, if $E=F/[R,F]$ and $D=\Delta/[S,\Delta]$, by identifying $D$ with $F/[R,F]T$, we obtain an extension $1\to Z\to E\to D\to 1$ with cyclic kernel $Z=\gen{[R,F]f_1\cdots f_{d+1}}$.
On the other hand, every elliptic presentation can be obtained in this way, by Theorem \ref{Satz:Hopf properties of the covers} we have the following result.
\begin{corollary}
Every elliptic central extension of a finite group is finite, and every elliptic central extension of a $p$-group is a $p$-group.
\end{corollary}

It is natural to ask whether an elliptic central extension is a cover. In general the answer is negative, as we mentioned in the introduction, since $H_2\Delta$ is not necessarily trivial (to see that it is cyclic, follows immediately by the Hopf formula for  \eqref{eq:presenting Delta}).
Still, an elliptic presentation affording a cover always exists.
\begin{theorem}\label{Satz:existence of elliptic covers}
 Every non--trivial finite group admits an elliptic cover.
\end{theorem}
\begin{proof}
 Choose a hyperiodic presentation $1\to R\to F\to G\to 1$ by $F=\ast_{i=1}^d\gen{f_i}$, where $\pi(f_i)=g_i\neq 1$ for all $i$.
 We consider the presentation, which is also hyperiodic, obtained by cloning the generators
 \[1\to \tilde R\to\tilde F\to G\to 1\ \ ,\ \ \tilde F=\Ast_{i=1}^d\Ast_{j=1}^{m_i}\gen{f_{i,j}}\ \ ,\ \ \gen{f_{i,j}}\simeq\Z_{m_i}\ \ ,\ \ f_{i,j}\mapsto g_i\ .\]
 Define $f_\ast=f_{1,1}\cdots f_{1,m_1}\cdots f_{d,1}\cdots f_{d,m_d}$, since $\pi(f_\ast)=g_1^{m_1}\cdots g_d^{m_d}=1$ we otain, as in \eqref{eq:hyperiodic presentation}, an elliptic presentation
\[1\to\tilde S\to\tilde \Delta\to G\to 1\ \ ,\ \ \tilde\Delta=\tilde F/\gen{f_\ast}^{\tilde F}\ .\]
On the other hand, the element $f_\ast$ lies in the kernel of the surjective homomorphism from $\tilde F$ to $F$ mapping $f_{i,j}$ to $f_i$ for all $i,j$,  which in turn factorizes through $\tilde\Delta$.
It follows that the central extension $\tilde D=\tilde\Delta/[\tilde\Delta,\tilde S]$ maps surjectively onto $E=F/[R,F]$, and therefore it is a cover.
\end{proof}
At this point we describe the underlying topology in analogy with Hopf's proof of the formula, referring to \cite{Brown,Spanier}.
A \emph{covering projection} between topological spaces $W$ and $Y$ is a continuos map $p:W\to Y$ such that every point $y\in Y$ admits a neighborhood $U$ which is \emph{evenly covered} by $p$, that is to say, $\inv{p}(U)$ is the disjoint union of open subsets of $W$ each one mapped homeomorphically onto $U$ by $p$.
In this situation, $W$ is a \emph{covering space} of $Y$, and we have the \emph{group of deck transformations}
\[\Aut(W/Y)=\{\ \varphi\in\Aut(W)\ |\ p\circ\varphi=p\ \}\ .\]
Covering spaces satisfy the \emph{unique--lifting property}, in particular, every path $\gamma$ in $Y$ can be lifted to a path $\gamma'$ in $W$, and if $\gamma''$ is another lifting of $\gamma$ which agrees with $\gamma'$ on some point, then $\gamma''=\gamma'$.
If $Y$ is connected, locally simply--connected, and semilocally path--connected, then there exists the \emph{universal cover} $\tilde Y$, which is a simply connected covering space, and in this case the fundamental group of $Y$ can be obtained as
\[\pi_1Y=\Aut(\tilde Y/Y)\ .\]
This is the case of connected locally--finite cellular  and simplicial complexes, in this case the universal cover $\tilde Y$ has a natural cellular structure, where a cell $\hat\sigma$ in $\tilde Y$ is a connected components of $\inv{p}(\sigma)$ for some cell $\sigma$ of $Y$, and this structure is preserved by $\Aut(\tilde Y/Y)$.
A cellular complex $Y$ affords the chain complex
\[\ZY_\ast:\ \cdots\to\ZY_2\to\ZY_1\to\ZY_0\to 0\] where $\ZY_k$ is the free $\Z$-module over $Y_k$, the homomorphisms $\partial:\ZY_q\to\ZY_{q-1}$ are determined by the topological boundary map $\partial:Y_q\to(Y_{q-1})^{q}$, and they satisfy $\partial^2=0$.
The \emph{homology} $H_\ast Y$ is the collection of abelian groups \[H_qY=\ker\partial_q/\im\partial_{q+1}\ .\]
The map $\partial$ is particularly easy to describe for a simplicial complex $Y$: for an oriented $q$-simplex $\sigma=\gen{s_0,\ldots,s_q}$, the $i$-th face is $\sigma_i=\gen{s_0,\ldots,\hat s_i,\ldots,s_{q}}$, and the topological boundary $\partial\sigma=(\sigma_0,\ldots,\sigma_{q})$ corresponds to the element $\partial\sigma=\sum_{i=1}^q(-1)^i\sigma_i$ in the chain module $\ZY_{q-1}$.
Since the $q$-simplices are free generators of $\ZY_q$, the map $\partial$ extends to a homomorphism by linearity.
The singular homology addresses the case of generic topological spaces, still giving the same result for cellular and simplicial complexes.
In general the homology is a homotopical invariant, thus a contractible space $Y$ has the same homology of a point, that is $H_0Y=\Z$ and $H_qY=0$ for $q\geq 1$.
Moreover, if $Y$ is connected then $H_0Y=\Z$, and if it is path--connected then
\[H_1Y\simeq\pi_1Y/[\pi_1Y,\pi_1Y]\ .\]

A \emph{$G$-complex} is a cellular or simplicial complex $Y$ endowed with a topological action $G\to\Aut(Y)$, in such a way that $G$ permutes the cells.
In this situation $\ZY_\ast$ is a chain complex of permutation $\ZG$-modules, and every homology group $H_qY$ is a $\ZG$-module.
To describe a $G$-complex we list in every dimension $Y_q$ the representatives for the action together with their stabilizer in $G$ and their boundary in $Y_{q-1}$.
It is worth to illustrate some elementary examples.

\emph{i) The action
of $\Z$ by translation on the real line $\R$.}
We decompose the real line in a one--dimensional complex whose vertices are the points $k\in\Z\subseteq\R$ and edges are the intervals $[k-1,k]\subseteq\R$.
By identifying $\Z$ with an infinite cyclic group $G=\gen{g}\simeq C_\infty$, the complex $Y$ is given by
\[Y_0=G\v\ \ ,\ \ Y_1=G\e\ \ ;\ \ \partial\e=(g\v,\v)\]
so that $g^k\v$ and $g^k\e$ correspond in the real line $\R$ to $k$  and $[k-1,k]$ respectively, in this case the stabilizers are trivial (the action is free).
The space $Y$ is a simplicial complex, and it is generated by the simplices $\gen{\v}$ and $\e=\gen{\v,g\v}$ under the action of $G$.
Therefore, $Y$ affords the chain complex
\[0\to\Z C_\infty\to\Z C_\infty\to 0\ \ ;\ \ 0\to\ZG\e\to\ZG\v\to 0\ \ ;\ \ \partial\e=g\v-\v\]
and the homology is $H_0Y=\Z$ and $H_1Y=0$, in fact $Y$ is contractible.

\emph{ii) The cyclic group $C_m=\gen{g}$ acting on the unit circle $\Sph^1\subseteq\C$ by rotations.}
We consider $Y$ to be an $m$-\emph{cycle}, the boundary of a polygon with $m$ edges, corresponding to the $m$-th cyclotomic subdivision of $\Sph^1$.
The definition of the complex $Y$ is similar to $i)$  but for a different group. The chain complex is
\[0\to\Z C_n\to\Z C_n\to 0\ \ ;\ \ 0\to\ZG\e\to\ZG\v\to 0\ \ ;\ \ \partial\e=g\v-\v\]
and we have $H_0Y=\Z$ and $H_1 Y=\gen{\nu\e}\simeq\Z$, where $\nu=1+g+\cdots+g^{m-1}$ is the norm element of $\ZG$.
In fact, $\Sph^1$ is connected and $\pi_1\Sph^1=\Z$.

\emph{iii) The Cayley graph of a group.}
For any group $G$ which is finitely generated by $g_\ast=(g_1,\ldots,g_d)$, the \emph{Cayley graph} $Y$ associated to $g_\ast$ is the one-dimensional complex defined by
\[Y_0=G\v\ \ ,\ \ Y_1=G[\e_1,\ldots,\e_d]\ \ ;\ \ \partial\e_i=(g_i\v,\v)\]
again the stabilizers are all trivial.
Therefore, $Y$ affords the chain complex
\[0\to(\ZG)^d\to\ZG\to 0\ \ ;\ \ 0\to\ZG[\e_1,\ldots,\e_d]\to\ZG\v\to 0\ \ ;\ \ \partial\e_i=g_i\v-\v\]
where $H_0Y=\Z$, and in general $H_1Y$ does not vanishes.

We remark that, for any surjective homomorphism $\pi:H\to G$ of groups, if $H$ is finitely generated by $h_\ast$ then $G$ is generated by $g_\ast=\pi h_\ast$.
Therefore, we have a natural map $p:W\to Y$, which is a covering projection between the Cayley graphs $W$ and $Y$ associated to $h_\ast$ and $g_\ast$ respectively.
The map $p$ is defined in the obvious way by setting $p(h\v)=(\pi h)\v$ and $p(h\e_i)=\pi(h)\e_i$.
Denoting by $A=\ker\pi$, this gives a homeomorphism $Y\simeq W/A$ which for the chain modules amounts to the isomorphism $\ZY_\ast\simeq\Z W_\ast\otimes_A\Z$ induced by $\pi$.

This is the starting point in the proof of the Hopf formula.
For a free presentation $1\to R\to F\to G\to 1$, the Cayley graph of $F$ is a tree $\tilde X$, that of $G$ is $Y=\tilde X/R$, and $X=\tilde X/F$ is a bouquet of circumferences.
Since $\tilde X$ is contractible, it is the universal cover of $X$ and $Y$, therefore
\[R\simeq\pi_1Y\ \ ,\ \ F\simeq\pi_1X\ \ ,\ \ G\simeq\Aut(Y/X)\ .\]
In particular, $H_1X\simeq F/[F,F]$, and $H_1Y\simeq R/[R,R]$ is the relations module.
It can be shown that $H_2G\simeq\ker\{(H_1Y)_G\to H_1X\}$, where $(H_1Y)_G=H_1Y\otimes_{\ZG}\Z$ are the coinvariants of $H_1Y$.
Since the action of $G$ is induced by conjugation of $F$ over $R$, we have $(H_1Y)_G\simeq R/[R,F]$, so that $H_2G\simeq R\cap[F,F]/[R,F]$.

For our purpose, first of all we shall find the analogue of a tree for a free product $F=\Ast_{i=1}^d\gen{f_i}$ of finite cyclic groups. 
The Cayley graph of $F$ is not contractible, we have seen this phenomena already for the case $d=1$ in the example $ii)$, it resembles the above tree $\tilde X$ in sense that it is the join of many cycles just as $\tilde X$ is the join of many lines.
We shall fill all these cycles to obtain a contractible space, first we describe the familiar case of a single cyclic group. 

\emph{iv) The cyclic group $C_m=\gen{g}$ acting on the unit disk $\Disk^2\subseteq\C$ by rotations}
\[C_n\to\Aut(\Disk^2)\ ,\ g^k:\zeta\to e^{2\pi ik/m}\zeta\]
As anticipated, we start from the Cayley graph $G[\v,\e]$ of $C_m$, and fill the cycle $\gamma=(\e,g\e,\ldots,g\e^{m-1})$ with a two--cell $\p$, a polygon with $m$-sides.
Therefore, in dimension two we have the only $\p$ whose stabilizer is the whole group, and $Y$ is  defined by
\[Y_0=G\v\ ,\ Y_1=G\e\ ,\ Y_2=\{\p\}\]
\[G_\p=G\ ;\ \partial\e=(g\v,\v)\ ,\ \partial\p=(\e,g\e,\ldots,g^{n-1}\e)\]
the boundary $\partial\p$ is attached to $Y_1$, by starting at $\v$ and following the orientation of the cycle.
The chain module $\ZY_2=\Z\p$ is the trivial module $\Z$, so that we obtain the chain complex
\[0\to\Z\to\Z C_n\to\Z C_n\to 0\ \ ;\ \ 0\to\Z\p\to\ZG\e\to\ZG\v\to 0\]
\[\partial\p=\nu\e\ \ ,\ \ \partial\e=g\v-\v\ \ ;\ \ \nu=1+g+\ldots+g^{n-1}\ .\]
Nonetheless, since the topological action on $\p$ is non-trivial we should describe this action in a second way, this time by associating a \emph{simplicial} complex $Y$ endowed with a simplicial action.
The new complex is obtained by barycentric subdivision of the polygon $\p$, therefore, we add a new vertex $\c$ which is the center of $\p$, the radii $g^k\r$ connecting $g^k\v$ to $\c$, and triangles $g^k\f$ subdividing $\p$, thus $Y$ is a \emph{wheel} filled by triangles.
Precisely,
\[Y_0=G[\v,\c]\ \ ,\ \ Y_1=G[\e,\r]\ \ ,\ \ Y_2=G[\f]\ \ ;\ \ G_\c=G\]
\[\partial\e=(g\v,\v)\ \ ,\ \ \partial\r=(\c,\v)\ \ ,\ \ \partial\f=(g\r,\r,\e)\]
and the generating simplices are
\[\gen{\v}\ \ ,\ \ \gen{\c}\ \ ,\ \ \e=\gen{\v,g\v}\ \ ,\ \ \r=\gen{\v,\c}\ \ ,\ \ \f=\gen{\v,g\v,\c}\]
Therefore, we obtain the chain modules and homomorphisms
\[0\to\Z C_n\to(\Z C_n)^2\to\Z C_n\oplus\Z\to 0\ \ ;\ \ 0\to\ZG\f\to\ZG[\e,\r]\to\ZG[\v,\c]\to 0\]
\[\partial\f=g\r-\r+\e\ \ ,\ \ \partial\e=g\v-\v\ \ ,\ \ \partial\r=\c-1\ .\]
The homology of these two complexes is $H_0Y=\Z$, and $H_1Y=H_2Y=0$, in fact, the underlying space is $\Disk^2$ which is contractible.
We remark, in addition, that the simplicial complex $Y$ retracts over the graph $\ZG[\v,\c,\r]$, which is a \emph{star} of center $\c$ with radii $\r,g\r,\ldots,g^{m-1}\r$ ending at $\v,g\v,\ldots,g^{m-1}\v$.

Readily, we define the complexes associated to a periodic generating system of a finitely generated group $G$, by giving the representatives, the non--trivial stabilizers, and the boundaries maps.
\begin{definition}\label{dfn:periodic complexes}
Let $G$ be a group which admits a periodic generating system
\[g_\ast=(g_1,\ldots,g_d)\ \ ,\ m_i=o(g_i)\ ,\ 1<m_i<\infty\ .\]
The \emph{periodic complex} is the simplicial complex defined by
\[Y_0=G[\v,\c_1,\ldots,\c_d]\ ,\ Y_1=G[\e_1,\ldots,\e_d,\r_1,\ldots\r_d]\ ,\ Y_2=G[\f_1,\ldots,\f_d]\]
\[G_{\c_i}=\gen{g_i}\ \ ;\ \ \partial\e_i=(g_i\v,\v)\ \ ,\ \ \partial\r_i=(\c_i,\v)\ \ ,\ \ \partial\f_i=(g_i\r_i,\r_i,\e_i)\]
where the generating simplices are
\[\gen{\v}\ ,\ \gen{\c_i}\ ,\ \e_i=\gen{\v,g_i\v}\ ,\ \r_i=\gen{\v,\c_i}\ ,\ \f_i=\gen{\v,g_i\v,\c_i}\ .\]
The \emph{periodic graph} is $G[\v,\c_1,\ldots,\c_d,\r_1,\ldots,\r_d]\subset Y$.
\end{definition}
In order to visualize $Y$ it is convenient to introduce, as in the example $iv)$, also the cellular complex $G[\v,\e_1,\ldots,\e_d,\p_1,\ldots,\p_d]$ where
\[\p_i=\f_i\cup g_i\f_i\cup\ldots\cup g_i^{m_i-1}\f_i\ \ ,\ \ \partial\p_i=(\e_i,g_i\e_i,\ldots,g_i^{m_i-1}\e_i)\ .\]
Thus, each vertex $g\v$ is the common vertex of the polygons $g\p_1,\ldots,g\p_d$, where $g\p_i$ has $m_i$ sides corresponding to the left coset $g\gen{g_i}$.
Moreover, $Y$ retracts onto the periodic graph, since $\f_i$ is the unique 2-simplex containing $\e_i$ and so it retracts on $\r_i\cup g_i\r_i$.
In the periodic graph the vertex $g\c_i=g\gen{g_i}\c_i$ is the center of the star $g\alpha_i$, where $\alpha_i=\gen{g_i}\r_i$, and each vertex $g\v$ is a common vertex of the stars $g\alpha_1,\ldots,g\alpha_d$, and these are the only stars containing it.

We denote $G_i=G/\gen{g_i}$, the permutation module $\ZG_i$ is induced from the trivial $\Z\gen{g_i}$-module $\Z$, so that
\[\ZG_i=\Z\uparrow_{\gen{g_i}}^G=\ZG\otimes_{\Z\gen{g_i}}\Z\ .\] 
The periodic complex $Y$ affords the chain complex of modules
 \[0\to(\ZG)^d\to(\ZG)^{2d}\to\ZG\oplus\ZG_1\oplus\cdots\oplus\ZG_d\to 0\]
\[0\to\ZG[\f_i]_i\to\ZG[\e_i,\r_i]_i\to\ZG[\v,\c_i]_i\to 0\]
\[ \partial\f_i=g_i\r_i-\r_i+\e_i\ \ ,\ \ \partial\e_i=g_1\v-\v\ \ ,\ \ \partial\r_i=\c_i-\v\ .\]
Similarly, the cellular complex affords the chain complex 
 \[0\to\ZG_1\oplus\cdots\oplus\ZG_d\to(\ZG)^d\to\ZG\to 0\]
 where $\partial\p_i=\nu_i\e_i$ for the norm element $\nu_i=1+g_i+\ldots+g_i^{m_i-1}$ of $\Z\gen{g_i}$, and 
 the periodic graph affords the complex
 \[0\to(\ZG)^{d}\to\ZG\oplus\ZG_1\oplus\cdots\oplus\ZG_d\to 0\ .\]
Since these spaces are homotopy equivalent, the chain complexes have the same homology groups, clearly $H_0Y=\Z$ since $Y$ is connected, and $H_2Y=0$ since $Y$ retracts on a graph.
In analogy with the classical case, $H_1Y\simeq R/[R,R]$, is the relation module for the hyperiodic presentation $1\to R\to F\to G\to 1$ by $F=\Ast_{i=1}^d\gen{f_i}$ with $\pi(f_\ast)=g_\ast$.
Here $R\simeq\pi_1Y$ is a the fundamental group of a graph and so it is free (alternatively this can be proved directly by using Kurosh Theorem, cf.~\cite{Robinson}).
In fact, we shall see that such a presentation induces a covering projection $\tilde Y\to Y$ where $\tilde Y$ is the periodic complex associated to $f_\ast$ (Theorem \ref{Satz:hyperiodic and local homeomorphism}), and the following result proves that $\tilde Y$ is the universal cover of $Y$ and in particular $R\simeq\pi_1Y$.
 \begin{theorem}\label{Satz:hyperiodic universal}
  Let $F=\Ast_{i=1}^d\gen{f_i}$ be a free product of finite cyclic groups.
  Then the periodic complex $\tilde Y$ associated to the generating system $f_\ast$ is contractible.
 \end{theorem}
 \begin{proof}
  Since $\tilde Y$ retracts on the periodic graph $\Gamma$ of $F$, it is enough to prove that $\Gamma$ is a tree.
  As a general property of free products, every element of $F$ can be written uniquely in reduced form as a product $w=u_1\cdots u_l$ where $1\neq u_i\in\gen{f_{j_i}}$ and $j_i\neq j_{i+1}$ for all $i=1,\ldots,l$, here the empty word $w=\varnothing$ corresponds to the identity element $1$ of $F$ \cite{Robinson}.
  The graph $\Gamma$ is bipartite in the set of vertices $F[\v]$ and $F[\c_1,\ldots,\c_d]$, therefore a non--trivial path $\gamma$ crosses at least one vertex $g\v$ in the component $F[\v]$.
  In particular, a non--trivial closed path $\gamma$ determines a sequence of vertices \[(g\v,g\c_{j_1},gu_1\v,gu_1\c_{j_2},gu_1u_2\v,\ldots,gu_1\cdots u_{l-1}\c_{j_l},gu_1\cdots u_l\v)\]
  where $u_i\in\gen{f_{j_i}}$ for all $i$, and $u_1\cdots u_l=1$, and when $\gamma$ is a cycle we have the additional condition that $u_i\neq 1$ and $j_i\neq j_{i+1}$, since $gu_1\cdots u_i\v\neq gu_1\cdots u_{i-1}\v$ and $gu_1\cdots u_i\c_{j_{i+1}}\neq gu_1\cdots u_i\c_{j_{i}}=gu_1\cdots u_{i-1}\c_{j_{i}}$.
  However, these are precisely the condition for the word $w=u_1\cdots u_l$ to be reduced, and since $u_1\cdots u_l=1$ in $F$ we obtain a contraddiction.
  It follows that $\Gamma$ does not contain any cycle, and thus it is a tree.
 \end{proof}
 We shall relate this construction to the periodic presentations and the associated finite covers, and it is not costly to work with more general extensions.
 Let $1\to N\to H\to G\to 1$ be any group extension, and denote $\pi:H\to G$.
 A generating system $h_\ast=(h_1,\ldots,h_d)$ of $H$ determines $g_\ast=\pi h_\ast$ of $G$, that is $g_\ast=(g_1,\ldots,g_d)$ where $g_i=\pi(h_i)$ for all $i$.
 Clearly, since $\pi$ is surjective and $o(g_i)\leq o(h_i)<\infty$ then $g_\ast$ is a periodic generating system.
 We say that the extension is \emph{periodic} in case $h_\ast$ is such and $o(g_i)>1$ for all $i$, that is, none of $h_1,\ldots,h_d$ is contained in $N$.
 Given a periodic extension $(H,G,h_\ast,g_\ast)$ we obtain by Definition \ref{dfn:periodic complexes} the simplicial complexes $W$ and $Y$, for $h_\ast$ and $g_\ast$ respectively, and there is a natural continuous simplicial map  $p:W\to Y$ associated to $\pi$.
 We define $p$ in the obvious way by setting
 \[p:W\to Y\ \ ,\ \ p(h\sigma)=\pi(h)\sigma\]
 where the symbol $\sigma\in\{\v,\c_i,\e_i,\r_i,\f_i\}$ denotes any generating simplex, respectively of $W$ and $Y$.
 Simply by restriction, we also have a map between the periodic graphs.
 Moreover, an analogue map can be defined for the cellular complexes. In this case, beside that $p(h\p_i)=g\p_i$, we must describe the topological map $p:h\p_i\to g\p_i$ by means of the barycentrical subdivision $\p_i=\f_i\cup g_i\f_i\cup\ldots g_i^{m_i-1}\f_i$ and the simplicial map.
 Alternatively, we can use the homeomorphism $\p_i\simeq\Disk^2$ and the power--map $\cdot~^{n_i}:\Disk^2\to\Disk^2$ for $n_i=o(h_i)/o(g_i)$.
 In case $n_i=1$ we have that $p:h\p_i\to g\p_i$ for $g=\pi(h)$ is a homeomorphism, otherwise $g\c_i$ is a \emph{branch point} and $p$ is \emph{ramified} at $g\c_i$.
 
 In all the three possibilities of simplicial complexes, cellular complexes, and graphs, for a \emph{hyperiodic extension}, that is when $o(h_i)=o(g_i)$ for all $i$, the map $p$ is a covering projection.
 On the other hand, if $o(h_i)>o(g_i)$ for some $i$,
 then the map $p$ is not a local homeomorphism, since the for any open neighborhood $U$ of $\c_i$ in $W$, the restriction of $p$ to $U$ is not injective.
 Therefore, we have obtained the following characterization for the hyperiodic extensions in terms of their topological action:
 \begin{theorem}\label{Satz:hyperiodic and local homeomorphism}
 A periodic extension $(H,G,h_\ast,g_\ast)$ induces a continous map $p:W\to Y$ between topological spaces.
 Moreover, $p$ is a covering projection if and only if the extension is hyperiodic.
 \end{theorem}
 We are interested in the above presentations related to the Hopf formula.
We let $1\to R\to F\to G\to 1$ be a periodic presentation of a finite group, $E=F/[R,F]$ be the associated periodic cover, and we suppose that the free generators $f_1,\ldots,f_d$ are not contained in $R$.
 By setting $h_i=[R,F]f_i$ and $g_i=Rf_i$, we obtain the complexes $\tilde W$, $W$, and $Y$ associated to the generating sets $f_\ast$, $h_\ast$ and $g_\ast$.
 We observe that, since $E/[E,E]\simeq F/[F,F]$, then the periods of $f_\ast$ and $h_\ast$ are the same, and so the map $\tilde W\to W$ is a covering projection.
 On the other hand, by the above theorem the map $W\to Y$ is a covering projection precisely when $E$ and $G$ have the same periods, that is when $E$ is hyperiodic.
 We have the following topological interpretation of the hyperiodic covers, where we see that being of minimal exponent is a necessary condition to afford a covering projection between cellular complexes.
 \begin{corollary}
 Any periodic cover $E$ of a finite group $G$ induces a continuous map $p:W\to Y$ between simplicial complexes, which is a covering projection if and only if the cover $E$ is hyperiodic.
 \end{corollary}
 Now we associate an elliptic generating system $g_\ast$ of a  group $G$ to an oriented simplicial surface $Y$.
 \begin{definition}
  Let $g$ be a group having an elliptic generating system
  \[g_\ast=(g_1,\ldots,g_{d})\ \ ,\ \ g_1\cdots g_{d+1}=1\ \ ,\ \ m_i=o(g_i)\ ,\ 1<m_i<\infty\ .\]
  The \emph{elliptic surface} of $G$ associated to $g_\ast$ is the simplicial complex
 defined by 
\[Y_0=G[\v,\c_i,\o]_i\ ,\ Y_1=G[\e_i,\r_i,\s_i]_i\ ,\ Y_2=G[\f_i,\t_i]_i\ \ ;\ \ i=1,\ldots,d+1\]
\[G_{\c_i}=\gen{g_i}\ \ ;\ \ \partial\e_i=(g_i\v,\v)\ \ .\ \ \partial\r_i=(\c_i,\v)\ \ ,\ \ \partial\s_i=(g_{i}\cdots g_{d+1}\o,\v)\]
\[\partial\f_i=(g_i\r_i,\r_i,\e_i)\ \ ,\ \ \partial\t_i=-(g_i\s_{i+1},\s_i,\e_i)\]
 the generating simplices are
 \[\gen{\v}\ ,\ \gen{\c_i}\ ,\ \gen{\o}\ ,\ \e_i=\gen{\v,g_i\v}\ ,\ \r_i=\gen{\v,\c_i}\ ,\ \s_i=\gen{\v,g_i\cdots g_{d+1}\o}\]
 \[\f_i=\gen{\v,g_i\v,\c_i}\ ,\ \t_i=\gen{\v,g_i\cdots g_{d+1}\o,g_i\v}\ .\]
 The reason for the negative sign of $\partial\t_i$ is to have a compatible orientation.
 \end{definition}
We describe how $Y$ is obtained from the simplicial complex $X$ associated to the periodic generating system $(g_1,\ldots,g_{d+1})$.
We observe that the path \[\xi=(\e_1,g_1\e_2,g_1g_2\e_3,\ldots,g_1\cdots g_d\e_{d+1})\]
is a cycle starting at $\v$, precisely because $g_1\cdots g_{d+1}=1$.
Thus, we attach a polygon $\q$ with $d+1$ sides to $\xi$, by following the opposite orientation (accordingly with the orientation of $X_2$).
To turn $X\cup G\q$ into a simplicial complex, we operate the barycentric subdivisions of $\q$ thus the vertex $\o$ is the center of $\q$, the edges $g_1\cdots g_{i-1}\s_i$ link $\o$ to $g_1\cdots g_{i-1}\v$, and for $i=1,\ldots,d+1$ the triangles $g_1\cdots g_{i-1}\t_i=\gen{g_1\cdots g_{i-1}\v,\o,g_1\cdots g_{i}\v}$ fill the polygon
\[\q=\t_1\cup g_1\t_2\cup\ldots\cup g_1\cdots g_d\t_{d+1}\ .\]
Similarly, we can obtain a cellular decomposition of $Y$ starting from  the periodic graph $G[\v,\c_i,\r_i]_i$,
and attaching a polygon $g\b$ with $2d+2$ sides to each cycle $g\xi$ where
\[\xi=(\r_1,-g_1\r_1,g_1\r_2,-g_1g_2\r_2,\ldots,g_1\cdots g_d\r_{d+1},-\r_{d+1})\ .\]
Thus $Y_0=G[\v,\c_1,\ldots,\c_{d+1}]$, $Y_1=G[\r_1,\ldots,\r_{d+1}]$, and $Y_2=G\b$ for
\[\b=\t_1\cup\f_1\cup g_1\f_1\cup g_1\t_2\cup g_1\f_2\cup g_1g_2\f_2\cup\ldots\cup g_1\cdots g_d\t_{d+1}\cup g_1\cdots g_d\f_{d+1}\cup\f_{d+1}\ .\]
It is important to observe that in the theory of Riemann surfaces the cell $\b$ correspond to a fundamental polygon for a canonical tassellation \cite{Magnus}.
\begin{figure}
 \begin{center}
   \includegraphics[width=.71\textwidth]{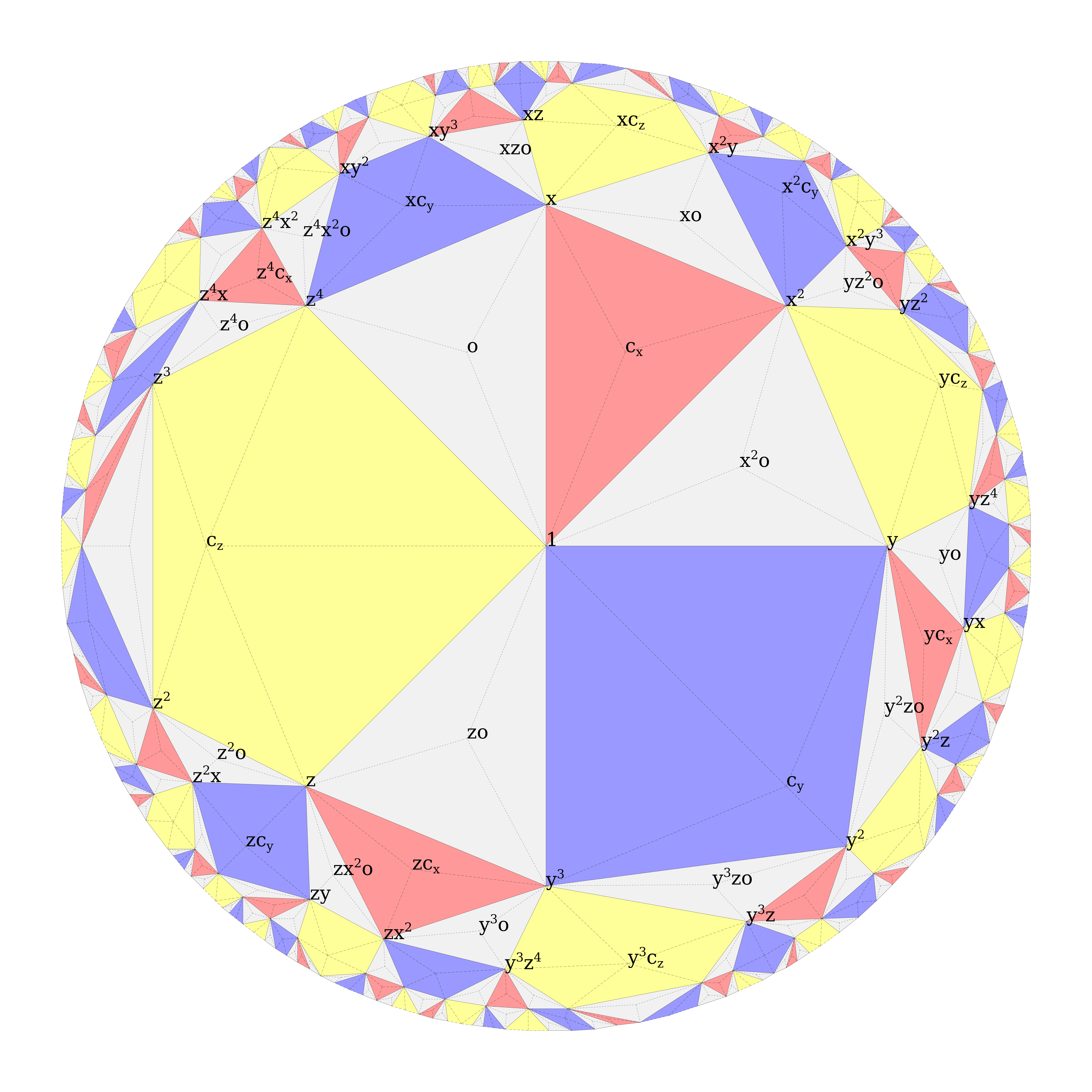}
 \end{center}
 \caption{A portion of the simplicial surface associated to the group $\Delta(3,4,5)$, where
 we denote $y_\ast=(x,y,z)$ and identify $g$ with $g\v$.}
\end{figure}

We show that complex $Y$ is a surface.
First we observe that $\f_i\cap\t_i=\{\e_i\}$, $\f_i\cap g_i\f_i=\{\r_i\}$, and $\t_i\cap g_i\t_{i+1}=\{g_i\s_i\}$, with the exceptions $\f_i\cap g_i\f_i=\{\r_i,g_i\r_i\}$ when $m_i=2$, and $\t_1\cap g\t_{2}=\{g_1\s_1,\s_2\}$ when $d=1$, in such cases $\p_i$ or $\q$ is a polygon with two vertices (\emph{vescica pisces}).
For all the other pairs of faces, the intersection is either empty or a single vertex.
Of course, any point in the interior of a face has 
a neighborhood homeomorphic to an open disk and, since every edge is a common side of two triangles, the same applies to the points in the interior of an edge.
Finally, this is also the case for each vertex in $Y_0$, since $\c_i$ is the center of $\p_i$, $\o$ is the center of $\q$, and around $\v$ we encounter 
$\ldots$, $\s_i$, $\e_i$, $\r_i$, $\inv{g_i}\e_i$, $\s_{i+1}$, $\ldots$ in radial ordering, therefore, a neighborhood of $\v$ homeomorphic to $\mathring\Disk$ can be found in the faces $\ldots$, $\t_i$, $\f_i$, $\inv{g_i}\f_i$, $\inv{g_i}\t_i$, $ \t_{i+1}$, $\ldots$
A simple counting argument gives us the arithmetic invariants.
 \begin{lemma}
 In case the group $G$ is finite, the oriented surface $Y$ is compact of Euler characteristic
 \[\chi(Y)=|G|\left(\sum_{i=1}^{d+1}\frac{1}{m_i}-d+1\right)\]
 and genus $\genus(Y)=1-\chi(Y)/2$.
 \end{lemma}
 This formula appears in a much deeper shape in the theory of Riemann surfaces \cite{AhlforsSario,Breuer}.
 The signatures corresponding to genus 1 and 2 are very restricted.
 The spherical groups, for which $\genus(Y)=0$, are all finite and they consists of two infinite families, namely the cyclic groups $\Z_n$ and the dihedral groups $D_{2n}$, together with the groups $A_4$, $S_4$ and $A_5$ which occur as orientation preserving symmetry groups of platonic solids.
 Parabolic groups, for which $\genus(Y)=1$ so they act on the torus, can be collected in three classes, according to the signatures $(2,3,6)$, $(3,3,3)$, and $(2,2,2,2)$, corresponding to some regular tiling of the euclidean plane.
 In negative characteristic, that is for higher genus, we have finite groups acting on hyperbolic compact surfaces.
 
The surface $Y$ affords the chain complex of modules
\[0\to(\ZG)^{2(d+1)}\to(\ZG)^{3(d+1)}\to(\ZG)^{2}\oplus\ZG_1\oplus\ldots\oplus\ZG_{d+1}\to 0\]
\[0\to\ZG[\f_i,\t_i]_i\to\ZG[\e_i,\r_i,\s_i]_i\to\ZG[\v,\o,\c_i]_i\to 0\]
\[\partial\f_i=g_i\r_i-\r_i+\e_i\ ,\ \partial\t_i=-(g_i\s_{i+1}-\s_i+\e_i)\ ,\ \partial\e_i=g_i\v-\v\ ,\ \partial\r_i=\c_i-\v\ .\]
The cellular decomposition $G[\v,\e_i,\q,\p_i]_i$ affords the chain complex
\[0\to\ZG\oplus\ZG_1\oplus\ldots\oplus\ZG_{d+1}\to(\ZG)^{d+1}\to\ZG\to 0\]
\[\partial\q=-(\e_1+g_1\e_2+\cdots+g_1\cdots g_d\e_{d+1})\ \  ,\ \ \partial\p_i=\nu_i\e_i\]
and the cellular decomposition $G[\v,\c_i,\r_i,\b]$ associated with the Cayley graph of $g_\ast$ affords the complex
\[0\to\ZG\to(\ZG)^{d+1}\to\ZG\oplus\ZG_1\oplus\ldots\oplus\ZG_{d+1}\to 0\]
\[\partial\b=\r_1-g_1\r_1+g_1\r_2-g_1g_2\r_2+g_1g_2\r_3-\ldots+g_1\cdots g_d\r_{d+1}-\r_{d+1}\ ,\ \partial\r_i=\c_i-\v\ .\]
In homology we have $H_0Y=\Z$ since $Y$ is connected.
Also $H_1Y\simeq S/[S,S]$ is the relation module with respect to the elliptic presentation $1\to S\to\Delta\to G\to 1$ for $\Delta=\Delta(m_\ast)$ where $m_\ast$ is the signature of $g_\ast$, and $\pi(y_\ast)=g_\ast$.
We will prove that the surface $\tilde Y$ associated to $y_\ast$ is simply connected (Theorem \ref{Satz:elliptic universal}), thus the kernel $S$ is isomorphic with the surface group of genus $\genus=\genus(Y)$
\[\pi_1Y\simeq S=\gen{\ s_1,\ldots,s_{2\genus}\ |\ \prod_{i=1}^\genus[s_{2i},s_{2i+1}]\ }\ .\]
In addition, for finite groups we have that $H_2Y=\Z$, since $Y$ is a compact surface, in fact the cycle $\sum_{g\in G}g\b$ in $\ZY_2=\ZG\b$ generates $H_2Y$.
 
 We consider a group extension $1\to N\to H\to G\to 1$ where $H$ admits an elliptic generating system $h_\ast=(h_1,\ldots,h_{d+1})$, so that $G$ admits the generating system $g_\ast=\pi h_\ast$.
If $o(g_i)>1$ for all $i$, that is when none of $h_1,\ldots,h_{d+1}$ belongs to $N$, then also $g_\ast$ is elliptic. In this case we say that $(H,G,h_\ast,g_\ast)$ is a \emph{elliptic pair}, and we obtain a natural continuous map between surfaces $p:W\to Y$, by setting $p(h\sigma)=\pi(h)\sigma$
where the symbol $\sigma\{\v,\o,\c_i,\e_i,\r_i,\s_i,\t,\f_i\}$ denotes any of the generating simplices respectively in $W$ and $Y$.
Once again, the ramification at the vertex $g\c_i$ occurs precisely when $n_i=o(h_i)/o(g_i)>1$, and all the other points of $Y$ are evenly covered by $p$.
Thus, we have the following characterization.
 \begin{theorem}\label{Satz:elliptic covers and topology}
  Any elliptic pair $(H,G,h_\ast,g_\ast)$ determines a continuous map $p:W\to Y$ between orientable surfaces. This map is a covering projection if and only if $h_\ast$ and $g_\ast$ have the same signature, and otherwise it is a ramified covering projection.
 \end{theorem}
 The most interesting case for us is that of an elliptic central extension for some elliptic presentation of a finite group $G$.
 \begin{corollary}
  Any elliptic presentation of a finite group induces a covering projection between compact orientable surfaces.
 \end{corollary}
 \begin{theorem}\label{Satz:elliptic universal}
  Let $1\to S\to\Delta\to G\to 1$ be an elliptic presentation of a finite group $G$, with elliptic generating systems $y_\ast$, and $g_\ast$, so that $\pi(y_\ast)=g_\ast$.
  Then $(\Delta,G,y_\ast,g_\ast)$ induces the universal covering projection $p:\tilde Y\to Y$.
  In particular the surface $\tilde Y$ associated to $y_\ast$ is homeomorphic with the sphere $\Sph^2$ in case $\genus(Y)=0$, or with the plane $\R^2$ in case $\genus{Y}>0$.
 \end{theorem}
\begin{proof}
 Let $Y$ be the surface associated to $g_\ast$.
 Since $Y$ is compact and orientable, then its universal cover $\tilde Y$ is either the sphere or the plane.
 Since $\tilde Y$ is a covering space of $Y$, then every automorphism of $Y$ can be lifted to an automorphism of $\tilde Y$.
 We consider the homomorphism $\varphi:G\to\Aut(Y)$ 
 given by the action on the simplicial structure of $Y$, so that $\varphi(g):\sigma\to g\sigma$ for all $g\in G$ and $\sigma\in Y_k$ where $k=0,1,2$.
 Therefore we obtain a group $\tilde\Delta=\gen{\tilde y_1,\ldots,\tilde y_{d+1}}$ where $\tilde y_i$ is the lifting of $\varphi(g_i)$ in $\Aut(\tilde Y)$.
 The simplicial decomposition of $Y$ induces a simplicial decomposition of $\tilde Y$, on which $\tilde\Delta$ acts simplicially.
 Every combinatorial path in $\tilde Y$ is the lifting of a combinatorial path in $Y$.
 If we consider an oriented path $\tilde\xi$ starting at $\v$, in the only $\e_i$'s. Since $p(\tilde\xi)=(\e_{i_1},g_{i_1}\e_{i_2},\ldots,g_{i_1}\cdots g_{i_{l-1}}\e_{i_l})$, by the unique lifting property we have that $\xi=(\e_{i_1},\tilde y_{i_1}\e_{i_2},\ldots,\tilde y_{i_1}\cdots \tilde y_{i_{l-1}}\e_{i_l})$, in particular, the vertices of $\tilde Y$ lying over $g\v$ are all of the kind $\tilde y\v$, and the edges of $\tilde Y$ lying over $g\e_i$ are all of the kind $\tilde y\e$, for $\tilde y\in\tilde\Delta$.
 Similarly, by considering a path starting at $v$ and terminating over $g\o$ or $g\c_i$, we have that $\tilde Y_0=\tilde\Delta[\v,\o,\c_i]_i$ and $\tilde Y_1=\tilde\Delta[\e_i,\r_i,\s_i]_i$.
 Since the action of $\tilde\Delta$ is a simplical action, we conclude that $\tilde Y=\tilde\Delta[\v,\o,\c_i,\e_i,\r_i,\f_i,\t_i]_i$.
 The closed paths $(\e_i,g_i\e_i,\ldots,g_i^{m_i-1}\e_i)$ and $(\e_1,g_1\e_2,\ldots,g_{1}\cdots g_{d}\e_{d+1})$ are null homotopic in $Y$, so are their unique lifting in $\tilde Y$. It follows that $o(\tilde y_i)=o(g_i)=m_i$ for all $i$ and $y_1\cdots y_{d+1}=1$,  so that $\tilde Y$ is the surface associated to the elliptic generating system $\tilde y_\ast$ of $\tilde\Delta$.
 Finally, we use the universal property of $\tilde Y$ to prove that $\Delta\simeq\tilde\Delta$.
 If $W$ is the surface associated to the elliptic generating system $y_\ast$ of $\Delta$, since $\Delta$ maps surjectively onto $\tilde\Delta$ we have a covering projection $p:W\to\tilde Y$. Since $\tilde Y$ is the universal cover of $Y$, then $p$ must be a homeomorphism.
\end{proof}

 An elliptic presentation $1\to S\to\Delta\to G\to 1$ provides a profinite group $D_\infty=\varprojlim\Delta/[S,_k\Delta]$, which is naturally associated to an inverse system of covering projections between compact orientable surfaces.
 In this direction, Theorem \ref{Satz:profinite covers} extends to an existence theorem for profinite elliptic covers, but in this case it is necessary to begin with a presentation providing an elliptic cover, as in Theorem \ref{Satz:existence of elliptic covers}, after the first step the system continues with Schur covers.
 
 To conclude, we compute the character $\pi$ of the relation module $V=H_1Y$ for the periodic complex and the oriented surface associated to a finite group.
 We tensor over $\C$ to produce semisimple modules, and so we consider the complex $\CY_\ast=\C\otimes\ZY_\ast$:
$0\to\CY_2\to\CY_1\to\CY_0\to 0$.
 Denoting by $\pi_q$ the character afforded by $\CY_q$, we can compute $\pi$ by the formula
\[\pi=-\pi_2+\pi_1-\pi_0+1_G\ .\]
 Indeed, we have $\CY_1\simeq \ker\partial_1\oplus\im\partial_1\simeq\im\partial_2\oplus \C V\oplus\im\partial_1\simeq\CY_2\oplus \C V\oplus\im\partial_1$,
 so that $\pi=\pi_1-\pi_2-\vartheta$ where $\vartheta$ is the character afforderd by $\im\partial_1$.
 On the other hand, since $H_0Y=\Z$ is the trivial $\ZG$-module, then $\CY_0\simeq\im\partial_1\oplus\C$ so that $\im\partial_1$ affords $\vartheta=\pi_0-1_G$.
 
 In the case of the periodic cellular complex, we have the chain complex
 \[0\to\C\uparrow_{\gen{g_1}}^G\oplus\ldots\oplus\C\uparrow_{\gen{g_{d+1}}}^G\to(\CG)^d\to\CG\to 0\]
 so that the character afforded by $H_1Y\simeq R/[R,R]$ is
 \[\pi=(d-1)\rho-\sum_{i=1}^d1_{\gen{g_i}}^G+1_G\]
  where $\rho$ is the regular character of $G$.
  
  In the case of the orientable surface associated to an elliptic presentation, we have the chain complex
 \[0\to\CG^{2(d+1)}\to\CG^{3(d+1)}\to(\CG)^2\oplus\C\uparrow_{\gen{g_1}}^G\oplus\ldots\oplus\C\uparrow_{\gen{g_{d+1}}}^G\to 0\]
 so that, the character afforded by $V\simeq S/[S,S]$ is 
 the \emph{Lefschetz character}
 \[\eta=(d-1)\rho-\sum_{i=1}^{d+1}1_{\gen{g_{i}}}^G+1_G\]
 in turn $\eta$ satisfies $\eta(1)=\chi(\Delta/S)$, and $\eta(g)=1-|\Fix_Y(g)|$
 for all $g\neq 1$.
 By realizing $Y$ as a Riemann surface, the Lefschetz character is the sum $\chi+\bar{\chi}$ where $\chi$ is the \emph{Eichler character}, that is afforded by the action of $G$ on the space of differential forms $\mathcal{H}^1(Y)$ \cite{Breuer}.

\end{document}